\documentclass[11pt,a4paper,english]{article}

\usepackage[latin1]{inputenc}
\usepackage[sort,numbers]{natbib}
\usepackage{babel}
\usepackage[centertags]{amsmath}
\usepackage{enumerate}
\usepackage{amsfonts}
\usepackage{amssymb}
\usepackage{mathtools}
\usepackage{amsthm}
\usepackage{newlfont}
\usepackage{dsfont}
\usepackage{geometry}
\usepackage{mathrsfs}
\usepackage{authblk}
\usepackage{verbatim}
\usepackage[hidelinks]{hyperref}

\DeclareMathOperator*{\essinf}{ess\,inf}
\DeclareMathOperator*{\esssup}{ess\,sup}

\DeclareMathOperator*{\argmax}{arg\,max}

\newtheorem{theorem}{Theorem}[section]
\newtheorem{lemma}[theorem]{Lemma}

\newtheorem{proposition}[theorem]{Proposition}
\newtheorem{definition}[theorem]{Definition}
\newtheorem{remark}[theorem]{Remark}
\newtheorem{assumption}[theorem]{Assumption}

\newcommand{\msc}[1]{\textbf{MSC2010 Classification:} #1.}

\newcommand{\keywords}[1]{\textbf{Key words:} #1.}

\allowdisplaybreaks[1]

\makeatletter  
\@addtoreset{equation}{section}

\makeatother  
\begin{document}
\title{\textbf{
Dynamic programming for discrete-time finite horizon optimal switching problems with negative switching costs}\footnote{This research was partially supported by EPSRC grant EP/K00557X/1.}}

\author{Randall Martyr\footnote{School of Mathematics, University of Manchester, Oxford Road, Manchester M13 9PL, United Kingdom. email: \texttt{randall.martyr@postgrad.manchester.ac.uk}}}
\maketitle
\begin{abstract}
This paper studies a discrete-time optimal switching problem on a finite horizon. The underlying model has a running reward, terminal reward and signed (positive and negative) switching costs. Using the martingale approach to optimal stopping problems, we extend a well known explicit dynamic programming method for computing the value function and the optimal strategy to the case of signed switching costs.
\end{abstract}
\msc{93E20, 60G40, 91B99, 62P20}
\vspace{+4pt}

\noindent\keywords{optimal switching, real options, stopping times, optimal stopping problems, Snell envelope} 
\section{Introduction}\label{Section:Introduction}
The relatively recent papers \cite{Yushkevich2002,Guo2008b} have shown the connection between Dynkin games and optimal switching problems where negative switching costs are allowed. In particular, \cite{Guo2008b} proved that the value of the Dynkin game is equal to the difference of the value functions for the two-mode optimal switching problem. However, there is no rigorous derivation of the dynamic programming algorithm for computing the value function in the case of signed (positive and negative) switching costs. In this paper, we resolve this issue by using a martingale approach to the optimal switching problem in discrete time.
\subsection{Literature review}
There are relatively few theoretical results on the dynamic programming method for optimal switching in discrete time. The discrete-time optimal switching problem with multiple modes was used in \cite{Carmona2008,Gassiat2012} as an approximation to the solution of a continuous time problem. The dynamic programming algorithm advocated in those papers follows from the backward induction method for solving optimal stopping problems (Chapter~\MakeUppercase{\romannumeral 1}, Section 1.1 of \cite{Peskir2006}), and requires the existence of a system of Snell envelope processes which solve the continuous-time optimal switching problem. Furthermore, their arguments for proving the existence of these processes requires strictly positive switching costs (see \cite{Martyr2014b} for related work where negative switching costs are allowed). A similar backward induction formula for the value function of an optimal switching problem with two modes and strictly positive, constant switching costs was obtained in \cite{Tanaka1990} under general non-Markovian assumptions. The paper \cite{Yushkevich2002}, which studied the optimal switching problem with two modes in a Markovian model, obtained a different type of dynamic programming equation for the value function -- one which is more in the spirit of the Wald-Bellman equations (Chapter~\MakeUppercase{\romannumeral 1}, Section 1.2 of \cite{Peskir2006}).

\subsection{Aim and results}
This paper looks at optimal switching for a finite-horizon discrete-time model which has a running reward, terminal reward and allows for negative switching costs. Our solution method is based on the martingale approach to optimal stopping problems. We provide the discrete-time analogue of the verification theorem of \cite{Djehiche2009} established for the continuous-time problem, and also justify and extend the dynamic programming method of \cite{Carmona2008,Gassiat2012} for computing the value function and optimal strategy to the case of signed switching costs. Numerical examples which utilize these results can be found in \cite{Martyr2014c}.

\section{Discrete-time optimal switching}\label{Section:Discrete-Time-Optimal-Switching}
\subsection{Definitions}
\subsubsection{Probabilistic setup.}
Let $(\Omega,\mathcal{F},\mathsf{P})$ be a complete probability space which is given. The expectation operator with respect to $\mathsf{P}$ is denoted by $\mathsf{E}$, and the indicator function of a set or event $A$ is written as $\mathbf{1}_{A}$. Let $\mathbb{T} = \lbrace 0,1,\ldots,T \rbrace$ represent a sequence of integer-valued times with $0 < T < \infty$. The probability space is equipped with a filtration $\mathbb{F} = (\mathcal{F}_{t})_{t \in \mathbb{T}}$, and it is assumed that $\mathcal{F}_{0}$ is the trivial $\sigma$-algebra, $\mathcal{F}_{0} = \lbrace \emptyset, \Omega \rbrace$, and $\mathcal{F} = \mathcal{F}_{T}$. The notation a.s. stands for ``almost-surely''. For a given $\mathbb{F}$-stopping time $\nu$, the notation $\mathcal{T}_{\nu}$ is used for the set of $\mathbb{F}$-stopping times $\tau$ such that $\nu \le \tau \le T$ \hspace{2bp} $\mathsf{P}$-a.s. Martingales, stopping times and other relevant concepts are understood to be defined with respect to the filtered probability space $(\Omega,\mathcal{F},\mathbb{F},\mathsf{P})$. The usual convention to suppress the dependence on $\omega \in \Omega$ is used below.

\subsubsection{Optimal switching definitions.}
The following data for the optimal switching problem are given:
\begin{enumerate}
	\item A set of operational modes $\mathbb{I} = \left\lbrace 1,2,\ldots,m \right\rbrace$, where $2 \le m < \infty$;
	\item A reward received at time $T$ for being in mode $i \in \mathbb{I}$, which is modelled by an $\mathcal{F}_{T}$-measurable real-valued random variable $\Gamma_{i}$;
	\item A running reward received while in mode $i \in \mathbb{I}$, which is represented by a real-valued adapted process $\Psi_{i} = \left(\Psi_{i}(t)\right)_{t \in \mathbb{T}}$;
	\item A cost for switching from mode $i \in \mathbb{I}$ to $j \in \mathbb{I}$, which is modelled by a real-valued adapted process $\gamma_{i,j} = \left(\gamma_{i,j}(t)\right)_{t \in \mathbb{T}}$.
\end{enumerate}

Define a class of \emph{admissible switching controls} as follows:
\begin{definition}\label{Definition:AdmissibleSwitchingControl}
	Let $t \in \mathbb{T}$ and $i \in \mathbb{I}$ be given. An admissible switching control starting from time $t$ in mode $i$ is a sequence $\alpha = \left(\tau_{n},\iota_{n}\right)_{n \ge 0}$ with the following properties: 
	
	\begin{enumerate}
		\item For $n \ge 0$, $\tau_{n} \in \mathcal{T}_{t}$ and satisfies $t = \tau_{0} \le \tau_{1} \le \tau_{2} \le \ldots \le T$; if $n \ge 1$ then $\mathsf{P}\left(\{\tau_{n} < T \} \cap \{\tau_{n} = \tau_{n+1}\}\right) = 0$.
		\item For $n \ge 0$, each $\iota_{n} \colon \Omega \to \mathbb{I}$ is $\mathcal{F}_{\tau_{n}}$-measurable; $\iota_{0} = i$ and $\iota_{n} \neq \iota_{n+1}$  $\mathsf{P}$-a.s. for $n \ge 0$.
	\end{enumerate}
	
	Let $\mathcal{A}_{t,i}$ denote the class of admissible switching controls (also called strategies) for the initial condition $(t,i) \in \mathbb{T} \times \mathbb{I}$.
\end{definition}

The switching control $\alpha = (\tau_{n},\iota_{n})_{n \ge 0}$ models the controller's decision to switch at time $\tau_{n}$, $n \ge 1$, from the active mode $\iota_{n-1}$ to another one $\iota_{n}$. The condition $\mathsf{P}\left(\{\tau_{n} < T \} \cap \{\tau_{n} = \tau_{n+1}\}\right) = 0$ for $n \ge 1$ means there is at most one switch at each time. The above definition is similar to the one given in \cite[p.~145]{Yushkevich2002}.
\begin{definition}\label{Definition:Discrete-Time-OS-ModeIndicatorEtc}
	Associated with each $\alpha = \left(\tau_{n},\iota_{n}\right)_{n \ge 0} \in \mathcal{A}_{t,i}$ are the following objects:
	\begin{itemize}
		\item a mode indicator function ${\mathbf{u} \colon \Omega \times \mathbb{T} \to \mathbb{I}}$ defined by,
		\begin{equation}\label{eq:Discrete-Time-OS-ModeIndicator}
		\mathbf{u}_{s} = \sum\limits_{n \ge 0}\iota_{n}\mathbf{1}_{\lbrace \tau_{n} \le s < \tau_{n+1}\rbrace},\hspace{1em} s \ge t.
		\end{equation}
		\item the (random) total number of switches before $T$
		\begin{equation}\label{eq:OS-Discrete-Time-RandomNumberOfSwitches}
		N(\alpha) = \sum_{n \ge 1}\mathbf{1}_{\lbrace \tau_{n} < T \rbrace}
		\end{equation}
		\item the last mode switched to before $T$
		\begin{equation}\label{eq:OS-Discrete-Time-Last-Mode-Switched-To}
		\iota_{N(\alpha)} = \sum_{n \ge 0}\iota_{n}\mathbf{1}_{\lbrace \tau_{n} < T \rbrace}\mathbf{1}_{\lbrace \tau_{n+1} = T \rbrace}, \quad \text{a.s.}
		\end{equation}
	\end{itemize}
\end{definition}

Note that the mode indicator function $\mathbf{u}$ may jump immediately from $\iota_{0}(\omega) = i$ to $\iota_{1}(\omega) \neq i$ if $\tau_{1}(\omega) = t$.

\subsection{The optimal switching problem}
Define the following performance index for switching controls with initial mode $i \in \mathbb{I}$ at time $t \in \mathbb{T}$:
\begin{equation}\label{Definition:Discrete-TimeSwitchingControlObjective}
J(\alpha;t,i) = \mathsf{E}\left[\sum_{s=t}^{T-1}\Psi_{\mathbf{u}_{s}}(s) + \Gamma_{\iota_{N(\alpha)}} - \sum\limits_{n \ge 1}\gamma_{\iota_{n-1},\iota_{n}}\left(\tau_{n}\right)\mathbf{1}_{\lbrace \tau_{n} < T \rbrace} \biggm \vert \mathcal{F}_{t}\right],\hspace{1em}\alpha \in \mathcal{A}_{t,i}
\end{equation}
The optimisation problem is to maximise the objective function $J(\alpha;t,i)$ over all admissible controls $\alpha \in \mathcal{A}_{t,i}$. The value function $V$ for the optimal switching problem is defined as a random function of the initial time and mode $(t,i)$:
\begin{equation}\label{eq:SwitchingControlValueFunction}
V(t,i) = \esssup\limits_{\alpha \in \mathcal{A}_{t,i}} J(\alpha;t,i).
\end{equation}
A switching control $\alpha^{*} \in \mathcal{A}_{t,i}$ is said to be optimal if it achieves the essential supremum in equation~\eqref{eq:SwitchingControlValueFunction}: $\mathsf{P}$-a.s.,
\begin{align*}
V(t,i) = J(\alpha^{*};t,i) \ge J(\alpha;t,i) \hspace{1 em}\forall \alpha \in \mathcal{A}_{t,i}.
\end{align*}

\begin{remark}
	Note that the analogous minimisation problem can be treated by negating the performance index~\eqref{Definition:Discrete-TimeSwitchingControlObjective} and value function~\eqref{eq:SwitchingControlValueFunction}. This follows from the relationship between the essential supremum and essential infimum: if $\Phi$ is a set of random variables on $(\Omega,\mathcal{F},\mathsf{P})$, the essential infimum of $\Phi$ with respect to $\mathsf{P}$ is defined as \cite[p.~496]{Follmer2011}:
	\[
	\essinf \Phi = \essinf_{\varphi \in \Phi} \varphi \coloneqq -\esssup_{\varphi \in \Phi}(-\varphi).
	\]
\end{remark}

\begin{remark}
	Processes or functions with super(sub)-scripts in terms of the mode indicators $\lbrace \iota_{n} \rbrace$ are interpreted in the following way:
	\begin{align*}
	Y^{\iota_{n}} & = \sum\limits_{j \in \mathbb{I}}\mathbf{1}_{\lbrace \iota_{n} = j \rbrace} Y^{j} \\
	\gamma_{\iota_{n-1},\iota_{n}}\left(\cdot\right) & = \sum\limits_{j \in \mathbb{I}}\sum\limits_{k \in \mathbb{I}}\mathbf{1}_{\lbrace \iota_{n-1} = j \rbrace}\mathbf{1}_{\lbrace \iota_{n} = k \rbrace} \gamma_{j,k}\left(\cdot\right)
	\end{align*}
	Note that the summations are finite.
\end{remark}

\subsection{Notation, conventions and assumptions}
The convention that $\sum_{s = v}^{t}\left(\cdot\right) = 0$ for any integers $t$ and $v$ with $t < v$ is used. The following terminology is referred to in later developments:
\begin{itemize}
	\item For a constant $p \ge 1$, let $L^{p}$ denote the class of random variables $Z$ satisfying $\mathsf{E}\left[\left|Z\right|^{p}\right] < \infty$.
	\item Similarly, let $\mathcal{S}^{p}$ denote the class of adapted processes $X$ satisfying $\mathsf{E}\left[\max\limits_{t \in \mathbb{T}}\left|X_{t}\right|^{p}\right] < \infty$.
\end{itemize}

\subsubsection{Assumptions.}
\begin{assumption}\label{Assumption:Optimal-Switching-Discrete-Time-Integrability}
	For each $i \in \mathbb{I}$, $\Gamma_{i} \in L^{2}$ and is $\mathcal{F}_{T}$-measurable,  $\Psi_{i} \in \mathcal{S}^{2}$, and for every $j \in \mathbb{I}$, $\gamma_{i,j} \in \mathcal{S}^{2}$.
\end{assumption}

Assumption~\ref{Assumption:Optimal-Switching-Discrete-Time-Integrability} ensures that the performance index \eqref{Definition:Discrete-TimeSwitchingControlObjective} is well defined and allows application of optimal stopping theory later. The following standard assumption on the switching costs \cite{Yushkevich2002,Guo2008b} is also imposed:
\begin{assumption}\label{assumption:SwitchingCosts}
	For every $i,j,k \in \mathbb{I}$ and $\forall t \in \mathbb{T}$: $\mathsf{P}$-a.s.,
	\begin{align}
	i. \quad & \gamma_{i,i}\left(t\right) = 0 \label{assumption:Discrete-Time-SwitchingCosts-No-Cost} \\
	ii. \quad & \gamma_{i,k}\left(t\right) < \gamma_{i,j}\left(t\right) + \gamma_{j,k}\left(t\right),\hspace{1em} \text{ if } i \neq j \text{ and } j \neq k\label{assumption:Discrete-Time-SwitchingCosts-No-Arbitrage}.
	\end{align}
\end{assumption}

Condition~\eqref{assumption:Discrete-Time-SwitchingCosts-No-Cost} says there is no additional cost for staying in the same mode. The second condition ensures that when going from one mode $i$ to another mode $k$, it is never profitable to immediately visit an intermediate mode $j$. By taking $k = i$ and using \eqref{assumption:Discrete-Time-SwitchingCosts-No-Cost}, condition~\eqref{assumption:Discrete-Time-SwitchingCosts-No-Arbitrage} also shows that it is unprofitable to switch immediately back and forth between modes.

\section{Optimal stopping and Snell envelopes}\label{Section:DiscreteSnellEnvelopes}
This section collects some results on discrete-time optimal stopping problems which are important below.

\begin{proposition}\label{Proposition:DiscreteSnellEnvelopeProperties}
	Let $U = (U_{t})_{t \in \mathbb{T}}$ be an adapted, $\mathbb{R}$-valued process that satisfies $U \in \mathcal{S}^{1}$. Then there exists an adapted, integrable $\mathbb{R}$-valued process $Z = (Z_{t})_{t \in \mathbb{T}}$ such that $Z$ is the smallest supermartingale which dominates $U$. The process $Z$ is called the Snell envelope of $U$ and it enjoys the following properties.
	\begin{enumerate}
		\item For any $t \in \mathbb{T}$, $Z_{t}$ is defined by:
		\begin{equation}
		Z_{t} = \esssup_{\tau \in \mathcal{T}_{t}}\mathsf{E}\left[U_{\tau} \vert \mathcal{F}_{t}\right].
		\end{equation}
		Moreover, $Z$ can also be defined recursively as follows: $Z_{T} \coloneqq  U_{T}$ and \\ $Z_{t} \coloneqq U_{t} \vee \mathsf{E}\left[Z_{t+1} \big \vert \mathcal{F}_{t}\right]$ for $t = T-1,\ldots,0$.
		\item For any $\theta \in \mathcal{T}$, the stopping time $\tau^{*}_{\theta} = \inf\lbrace t \ge \theta \colon Z_{t} = U_{t} \rbrace$ is optimal after $\theta$ in the sense that:
		\begin{equation}\label{eq:SnellEnvelopeAtStoppingTimes}
		Z_{\theta} = \mathsf{E}\bigl[U_{\tau^{*}_{\theta}} \big\vert \mathcal{F}_{\theta}\bigr] = \esssup_{\tau \in \mathcal{T}_{\theta}}\mathsf{E}\bigl[U_{\tau} \big\vert \mathcal{F}_{\theta}\bigr],\hspace{1em} \mathsf{P}-\text{a.s.}
		\end{equation}
		\item For any $t \in \mathbb{T}$ given and fixed, the stopped process $\left(Z_{r \wedge \tau^{*}_{t}}\right)_{t \le r \le T}$ is a martingale.
	\end{enumerate}
\end{proposition}
These results are standard and can be found in the references \cite{Neveu1975,Peskir2006,Follmer2011}.

\section{The verification theorem}\label{Section:Optimal-Switching-Discrete-Time-Verification-Theorem}
This section proposes a probabilistic solution to the optimal switching problem. The approach follows that of \cite{Djehiche2009} in continuous time, which postulates the existence of a particular system of $m$ stochastic processes and verifies (Theorem~\ref{theorem:verification}) that the components of this system solve the optimal switching problem with given initial conditions. The existence of these candidate optimal processes is proved in the following section (Theorem~\ref{theorem:ExistenceSnellEnevelope}).
\subsection{An iterative optimal stopping problem}
Suppose there exist $m$ real-valued, adapted processes $Y^{i} = \big(Y^{i}_{t}\big)_{t \in \mathbb{T}}$, $i \in \mathbb{I}$, such that $Y^{i} \in \mathcal{S}^{2}$ and
\begin{equation}\label{eq:VerificationSnellEnvelope}
Y^{i}_{t} = \esssup\limits_{\tau \in \mathcal{T}_{t}} \mathsf{E}\left[\sum_{s=t}^{\tau-1}\Psi_{i}(s) + \Gamma_{i}\mathbf{1}_{\lbrace \tau = T \rbrace} + \max\limits_{j \neq i} \left\lbrace Y^{j}_{\tau} - \gamma_{i,j}\left(\tau\right) \right\rbrace \mathbf{1}_{\lbrace \tau < T \rbrace} \biggm \vert \mathcal{F}_{t}\right].
\end{equation}

For $i \in \mathbb{I}$, define the \emph{implicit gain process} $\big(U^{i}_{t}\big)_{t \in \mathbb{T}}$ by,
\begin{equation}\label{eq:VerificationImplicitGainProcess}
U^{i}_{t} = \max\limits_{j \neq i} \left\lbrace Y^{j}_{t} - \gamma_{i,j}\left(t\right) \right\rbrace \mathbf{1}_{\lbrace t < T \rbrace} + \Gamma_{i}\mathbf{1}_{\lbrace t = T \rbrace}.
\end{equation}

Then equation~\eqref{eq:VerificationSnellEnvelope} becomes,
\begin{equation}\label{eq:VerificationSnellEnvelope2}
Y^{i}_{t} = \esssup\limits_{\tau \in \mathcal{T}_{t}}\mathsf{E}\left[\sum_{s=t}^{\tau-1}\Psi_{i}(s) + U^{i}_{\tau} \biggm \vert \mathcal{F}_{t}\right].
\end{equation}

Note that the assumptions on $Y^{i}$ and the costs guarantee that $U^{i} \in \mathcal{S}^{2}$ for every $i \in \mathbb{I}$. Recalling Proposition~\ref{Proposition:DiscreteSnellEnvelopeProperties}, $\left(Y^{i}_{t} + \sum_{s=0}^{t-1}\Psi_{i}(s)\right)_{t \in \mathbb{T}}$ can be identified as the Snell envelope of the process $\left(U^{i}_{t} + \sum_{s=0}^{t-1}\Psi_{i}(s)\right)_{t \in \mathbb{T}}$ (also see Lemma~\ref{lemma:VerificationLemma}).

\begin{theorem}[Verification Theorem]\label{theorem:verification}
	Let $i \in \mathbb{I}$ be the active mode at some fixed initial time $t \in \mathbb{T}$ and suppose $Y^{1},\ldots,Y^{m}$ as defined in equation~\eqref{eq:VerificationSnellEnvelope} are in $\mathcal{S}^{2}$. Define sequences of random times $\{\tau^{*}_{n}\}_{n \ge 0}$ and mode indicators $\{\iota^{*}_{n}\}_{n \ge 0}$ as follows:
	\begin{gather}\label{eq:OptimalStoppingStrategy}
	\begin{split}
	\tau^{*}_{0} & = t,\hspace{1pc} \iota^{*}_{0} = i \hspace{1 em} \text{and for } n \ge 1: \\
	\tau^{*}_{n} & = \inf\left\lbrace \tau^{*}_{n-1} \le s \le T \colon Y^{\iota^{*}_{n-1}}_{s} = U^{\iota^{*}_{n-1}}_{s}  \right\rbrace,\hspace{1em} \iota^{*}_{n} = \sum\limits_{j \in \mathbb{I}} j \mathbf{1}_{A^{\iota^{*}_{n-1}}_{j}}
	\end{split} \\
	\text{where } A^{\iota^{*}_{n-1}}_{j} (= A^{\iota^{*}_{n-1}}_{j}(\omega)) \text{ is the event:} \hspace{\textwidth} \nonumber \\
	A^{\iota^{*}_{n-1}}_{j} \coloneqq \left\lbrace Y^{j}_{\tau^{*}_{n}} - \gamma_{\iota^{*}_{n-1},j}\left(\tau^{*}_{n}\right) = \max\limits_{k \neq \iota^{*}_{n-1}}\left\lbrace Y^{k}_{\tau^{*}_{n}} - \gamma_{\iota^{*}_{n-1},k}\left(\tau^{*}_{n}\right) \right\rbrace \right\rbrace. \nonumber
	\end{gather}
	Then, $\alpha^{*} = \left(\tau^{*}_{n},\iota^{*}_{n}\right)_{n \ge 0} \in \mathcal{A}_{t,i}$ and satisfies
	\begin{equation}\label{eq:OptimalSwitchingProcess}
	Y^{i}_{t} = J(\alpha^{*};t,i) = \esssup\limits_{\alpha \in \mathcal{A}_{t,i}} J(\alpha;t,i)\hspace{1em} \text{a.s.}
	\end{equation}
\end{theorem}
\begin{proof}
	The proof is essentially the same as Theorem~1 of \cite{Djehiche2009}. Recall the definition of $U^{i}$ in equation~\eqref{eq:VerificationImplicitGainProcess}. At time $t$, $Y^{i}_{t}$ is given by
	
	\[
	Y^{i}_{t} = \esssup\limits_{\tau \in \mathcal{T}_{t}}\mathsf{E}\left[\sum_{s=t}^{\tau-1}\Psi_{i}(s) + U^{i}_{\tau} \biggm \vert \mathcal{F}_{t}\right].
	\]
	If $t = T$, then $\tau^{*}_{0} = T$, $\iota^{*}_{0} = i$ which leads to $Y^{i}_{T} = \Gamma_{i}$, and the claim follows trivially since $\Gamma_{i} = J(\alpha;T,i) = V(T,i)$ almost surely for any switching control $\alpha \in \mathcal{A}_{T,i}$.
	
	Suppose now that $t < T$. Lemma~\ref{lemma:AdmissibleOptimalControl} in the appendix verifies that $\alpha^{*} \in \mathcal{A}_{t,i}$. Note that the infimum in equation~\eqref{eq:OptimalStoppingStrategy} is always attained since $Y^{i}_{T} = U^{i}_{T}$ a.s. for every $i \in \mathbb{I}$. The stopping time $\tau^{*}_{1}$ in \eqref{eq:OptimalStoppingStrategy} is optimal after $t$ by Proposition~\ref{Proposition:DiscreteSnellEnvelopeProperties}. Using this together with the definition of $\iota^{*}_{1}$ gives, almost surely,
	\begin{align}\label{eq:Verification1}
	Y^{i}_{t} & = \esssup\limits_{\tau \in \mathcal{T}_{t}}\mathsf{E}\left[\sum_{s = t}^{\tau-1}\Psi_{i}(s) + U^{i}_{\tau} \biggm \vert \mathcal{F}_{t}\right] \nonumber \\
	& = \mathsf{E}\left[\sum_{s = t}^{\tau^{*}_{1} - 1}\Psi_{i}(s) + U^{i}_{\tau^{*}_{1}} \biggm \vert \mathcal{F}_{t}\right] \nonumber \\
	& = \mathsf{E}\left[\sum_{s = t}^{\tau^{*}_{1}-1}\Psi_{i}(s) + \Gamma_{i}\mathbf{1}_{\lbrace \tau^{*}_{1} = T \rbrace} + \max\limits_{j \neq i} \left\lbrace Y^{j}_{\tau^{*}_{1}} - \gamma_{i,j}\left(\tau^{*}_{1}\right)\right\rbrace \mathbf{1}_{\lbrace \tau^{*}_{1} < T \rbrace} \biggm \vert \mathcal{F}_{t}\right] \nonumber \\
	& = \mathsf{E}\left[\sum_{s = t}^{\tau^{*}_{1}-1}\Psi_{i}(s) + \Gamma_{i}\mathbf{1}_{\lbrace \tau^{*}_{1} = T \rbrace} + \left\lbrace Y^{\iota^{*}_{1}}_{\tau^{*}_{1}} - \gamma_{i,\iota^{*}_{1}}\left(\tau^{*}_{1}\right) \right\rbrace \mathbf{1}_{\lbrace \tau^{*}_{1} < T \rbrace} \biggm \vert \mathcal{F}_{t}\right]
	\end{align}
	
	Lemma~\ref{lemma:VerificationLemma} in the appendix confirms that $Y^{\iota^{*}_{1}}$ satisfies:
	\begin{equation}\label{eq:SequentialOptimality}
	Y^{\iota^{*}_{1}}_{s} = \esssup\limits_{\tau \in \mathcal{T}_{s}}\mathsf{E}\left[\sum_{r=s}^{\tau^{*}_{1}-1}\Psi_{\iota^{*}_{1}}(r) + U^{\iota^{*}_{1}}_{\tau} \biggm \vert \mathcal{F}_{s}\right] \enskip \text{on} \enskip [\tau^{*}_{1},T].
	\end{equation}
	
	Use \eqref{eq:SequentialOptimality} together with the definition and optimality of $\tau^{*}_{2}$ and $\iota^{*}_{2}$ to get, almost surely,
	\begin{align}
	\mathbf{1}_{\lbrace \tau^{*}_{1} < T \rbrace}Y^{\iota^{*}_{1}}_{\tau^{*}_{1}} = {} & \esssup\limits_{\tau \in \mathcal{T}_{\tau^{*}_{1}}}\mathsf{E}\left[\sum_{r=\tau^{*}_{1}}^{\tau - 1}\Psi_{\iota^{*}_{1}}(r) + U^{\iota^{*}_{1}}_{\tau} \biggm \vert \mathcal{F}_{\tau^{*}_{1}}\right]\mathbf{1}_{\lbrace \tau^{*}_{1} < T \rbrace} \nonumber \\
	= {} & \mathsf{E}\bigg[\sum_{r=\tau^{*}_{1}}^{\tau^{*}_{2}-1}\Psi_{\iota^{*}_{1}}(r) + \Gamma_{\iota^{*}_{1}}\mathbf{1}_{\lbrace \tau^{*}_{2} = T \rbrace} + \left\lbrace Y^{\iota^{*}_{2}}_{\tau^{*}_{2}} - \gamma_{\iota^{*}_{1},\iota^{*}_{2}}\left(\tau^{*}_{2}\right) \right\rbrace \mathbf{1}_{\lbrace \tau^{*}_{2} < T \rbrace} \biggm \vert \mathcal{F}_{\tau^{*}_{1}}\bigg]\mathbf{1}_{\lbrace \tau^{*}_{1} < T \rbrace}\label{eq:Verification2}
	\end{align}
	
	Combining~\eqref{eq:Verification1} and~\eqref{eq:Verification2} gives the following expression for $Y^{i}_{t}$: almost surely,
	\begin{align}\label{eq:Verification3}
	Y^{i}_{t} = {} & \mathsf{E}\left[\sum_{s = t}^{\tau^{*}_{1}-1}\Psi_{i}(s) + \Gamma_{i}\mathbf{1}_{\lbrace \tau^{*}_{1} = T \rbrace} + \left\lbrace Y^{\iota^{*}_{1}}_{\tau^{*}_{1}} - \gamma_{i,\iota^{*}_{1}}\left(\tau^{*}_{1}\right) \right\rbrace \mathbf{1}_{\lbrace \tau^{*}_{1} < T \rbrace} \biggm \vert \mathcal{F}_{t}\right] \nonumber \\
	= {} & \mathsf{E}\left[\sum_{s = t}^{\tau^{*}_{1}-1}\Psi_{i}(s) + \mathsf{E}\left[\sum_{r=\tau^{*}_{1}}^{\tau^{*}_{2}-1}\Psi_{\iota^{*}_{1}}(r) \biggm \vert \mathcal{F}_{\tau^{*}_{1}}\right] \mathbf{1}_{\lbrace \tau^{*}_{1} < T \rbrace} \biggm \vert \mathcal{F}_{t}\right] \nonumber \\
	& + \mathsf{E}\left[\Gamma_{i}\mathbf{1}_{\lbrace \tau^{*}_{1} = T \rbrace} + \mathsf{E}\left[\Gamma_{\iota^{*}_{1}}\mathbf{1}_{\lbrace \tau^{*}_{2} = T \rbrace}\big \vert \mathcal{F}_{\tau^{*}_{1}}\right] \mathbf{1}_{\lbrace \tau^{*}_{1} < T \rbrace} \big \vert \mathcal{F}_{t}\right] \nonumber \\
	& - \mathsf{E}\left[\gamma_{i,\iota^{*}_{1}}\left(\tau^{*}_{1}\right) \mathbf{1}_{\lbrace \tau^{*}_{1} < T \rbrace} + \mathsf{E}\left[ \gamma_{\iota^{*}_{1},\iota^{*}_{2}}\left(\tau^{*}_{2}\right) \mathbf{1}_{\lbrace \tau^{*}_{2} < T \rbrace} \big \vert \mathcal{F}_{\tau^{*}_{1}}\right] \mathbf{1}_{\lbrace \tau^{*}_{1} < T \rbrace} \big \vert \mathcal{F}_{t}\right] \nonumber \\
	& + \mathsf{E}\left[\mathsf{E}\left[Y^{\iota^{*}_{2}}_{\tau^{*}_{2}}\mathbf{1}_{\lbrace \tau^{*}_{2} < T \rbrace} \big \vert \mathcal{F}_{\tau^{*}_{1}}\right] \mathbf{1}_{\lbrace \tau^{*}_{1} < T \rbrace} \big \vert \mathcal{F}_{t}\right]
	\end{align}
	
	Since $\mathbf{1}_{\lbrace \tau^{*}_{1} = T \rbrace}$, $\mathbf{1}_{\lbrace \tau^{*}_{1} < T \rbrace}$ and $\gamma_{i,\iota^{*}_{1}}\left(\tau^{*}_{1}\right)$ are all $\mathcal{F}_{\tau^{*}_{1}}$-measurable, they can be brought inside the conditional expectation with respect to $\mathcal{F}_{\tau^{*}_{1}}$ in equation~\eqref{eq:Verification3}: almost surely,
	\begin{align*}
	Y^{i}_{t} = {} & \mathsf{E}\left[\mathsf{E}\left[\sum_{s = t}^{\tau^{*}_{1}-1}\Psi_{i}(s) + \sum_{r=\tau^{*}_{1}}^{\tau^{*}_{2}-1}\Psi_{\iota^{*}_{1}}(r)\mathbf{1}_{\lbrace \tau^{*}_{1} < T \rbrace} \biggm \vert \mathcal{F}_{\tau^{*}_{1}}\right] \biggm \vert \mathcal{F}_{t}\right] \\
	& + \mathsf{E}\left[\mathsf{E}\left[\Gamma_{i}\mathbf{1}_{\lbrace \tau^{*}_{1} = T \rbrace} + \Gamma_{\iota^{*}_{1}}\mathbf{1}_{\lbrace \tau^{*}_{2} = T \rbrace}\mathbf{1}_{\lbrace \tau^{*}_{1} < T \rbrace}\big \vert \mathcal{F}_{\tau^{*}_{1}}\right] \big \vert \mathcal{F}_{t}\right] \\
	& - \mathsf{E}\left[\mathsf{E}\left[\gamma_{i,\iota^{*}_{1}}\left(\tau^{*}_{1}\right) \mathbf{1}_{\lbrace \tau^{*}_{1} < T \rbrace} + \gamma_{\iota^{*}_{1},\iota^{*}_{2}}\left(\tau^{*}_{2}\right) \mathbf{1}_{\lbrace \tau^{*}_{2} < T \rbrace}\mathbf{1}_{\lbrace \tau^{*}_{1} < T \rbrace} \big \vert \mathcal{F}_{\tau^{*}_{1}}\right] \big \vert \mathcal{F}_{t}\right] \\
	& + \mathsf{E}\left[\mathsf{E}\left[Y^{\iota^{*}_{2}}_{\tau^{*}_{2}}\mathbf{1}_{\lbrace \tau^{*}_{2} < T \rbrace}\mathbf{1}_{\lbrace \tau^{*}_{1} < T \rbrace} \big \vert \mathcal{F}_{\tau^{*}_{1}}\right] \big \vert \mathcal{F}_{t}\right]
	\end{align*}
	
	Using $\iota^{*}_{0} = i$ and the definition of the mode indicator $\mathbf{u}^{*}$ (cf. \eqref{eq:Discrete-Time-OS-ModeIndicator}) gives,
	\[
	\sum_{s = t}^{\tau^{*}_{2}-1}\Psi_{\mathbf{u}^{*}_{s}}(s) = \sum_{r = t}^{\tau^{*}_{1}-1}\Psi_{i}(r) + \sum_{r=\tau^{*}_{1}}^{\tau^{*}_{2}-1}\Psi_{\iota^{*}_{1}}(r)\mathbf{1}_{\lbrace \tau^{*}_{1} < T \rbrace}.
	\]
	
	Lemma~\ref{lemma:AdmissibleOptimalControl} in the appendix verifies that $\mathsf{P}\left(\{\tau^{*}_{n} < T \} \cap \{\tau^{*}_{n} = \tau^{*}_{n+1}\}\right) = 0$ and ensures the above expression is well defined. Since $\tau^{*}_{1} \le \tau^{*}_{2}$ it follows that $\left\lbrace \tau^{*}_{2} < T  \right\rbrace \subset \left\lbrace \tau^{*}_{1} < T  \right\rbrace$ and therefore $\mathbf{1}_{\lbrace \tau^{*}_{2} < T \rbrace}\mathbf{1}_{\lbrace \tau^{*}_{1} < T \rbrace} = \mathbf{1}_{\lbrace \tau^{*}_{2} < T \rbrace}$ almost surely. Note that $\tau^{*}_{0} = t < T$ so $\mathbf{1}_{\lbrace \tau^{*}_{0} < T \rbrace} = 1$ and that $\mathcal{F}_{t} \subseteq \mathcal{F}_{\tau^{*}_{1}}$ since $t \le \tau^{*}_{1}$. These observations together with the tower property of conditional expectations, shows that $Y^{i}_{t}$ satisfies: almost surely,
	\begin{align}\label{eq:Verification4}
	Y^{i}_{t} = {} & \mathsf{E}\left[\mathsf{E}\left[\sum_{s = t}^{\tau^{*}_{2}-1}\Psi_{\mathbf{u}^{*}_{s}}(s) + \sum_{n = 1}^{2}\Gamma_{\iota^{*}_{n-1}}\mathbf{1}_{\lbrace \tau^{*}_{n} = T \rbrace}\mathbf{1}_{\lbrace \tau^{*}_{n-1} < T \rbrace} \biggm \vert \mathcal{F}_{\tau^{*}_{1}}\right] \biggm \vert \mathcal{F}_{t}\right] \nonumber \\
	& + \mathsf{E}\left[\mathsf{E}\left[\sum_{n = 1}^{2}-\gamma_{\iota^{*}_{n-1},\iota^{*}_{n}}\left(\tau^{*}_{n}\right)\mathbf{1}_{\lbrace \tau^{*}_{n} < T \rbrace} + Y^{\iota^{*}_{2}}_{\tau^{*}_{2}} \mathbf{1}_{\lbrace \tau^{*}_{2} < T \rbrace} \biggm \vert \mathcal{F}_{\tau^{*}_{1}}\right] \biggm \vert \mathcal{F}_{t}\right] \nonumber \\
	\begin{split}
	= {} & \mathsf{E}\left[\sum_{s = t}^{\tau^{*}_{2}-1}\Psi_{\mathbf{u}^{*}_{s}}(s) + \sum_{k = 0}^{1}\Gamma_{\iota^{*}_{k}}\mathbf{1}_{\lbrace \tau^{*}_{k} < T \rbrace}\mathbf{1}_{\lbrace \tau^{*}_{k+1} = T \rbrace} - \sum_{n = 1}^{2}\gamma_{\iota^{*}_{n-1},\iota^{*}_{n}}\left(\tau^{*}_{n}\right)\mathbf{1}_{\lbrace \tau^{*}_{n} < T \rbrace} \biggm \vert \mathcal{F}_{t}\right] \\
	& + \mathsf{E}\left[Y^{\iota^{*}_{2}}_{\tau^{*}_{2}} \mathbf{1}_{\lbrace \tau^{*}_{2} < T \rbrace} \biggm \vert \mathcal{F}_{t}\right].
	\end{split}
	\end{align}
	
	Let $N(\alpha^{*})$ be the total number of switches under $\alpha^{*}$ (cf. \eqref{eq:OS-Discrete-Time-RandomNumberOfSwitches}). Since $\alpha^{*}$ is admissible by Lemma~\ref{lemma:AdmissibleOptimalControl} in the appendix, for $n \ge 1$ the switching times satisfy $\tau^{*}_{n} < \tau^{*}_{n+1}$ on $\{n \le N(\alpha^{*})\}$ and $\tau^{*}_{n} = T$ on $\{n > N(\alpha^{*})\}$. Repeating the procedure of substituting for $Y^{\iota^{*}_{n}}_{\tau^{*}_{n}}$ with $ n = 2, 3,\ldots$ yields
	\begin{equation}\label{eq:Verification5}
	Y^{i}_{t} =
	\mathsf{E}\left[\sum_{s = t}^{T-1}\Psi_{\mathbf{u}^{*}_{s}}(s) + \sum_{k = 0}^{N(\alpha^{*})}\Gamma_{\iota^{*}_{k}}\mathbf{1}_{\lbrace \tau^{*}_{k} < T \rbrace}\mathbf{1}_{\lbrace \tau^{*}_{k+1} = T \rbrace} - \sum_{n = 1}^{N(\alpha^{*})}\gamma_{\iota^{*}_{n-1},\iota^{*}_{n}}\left(\tau^{*}_{n}\right)\biggm \vert \mathcal{F}_{t}\right].
	\end{equation}
	
	The sum of the terminal rewards collapses to a single term,
	\begin{equation}\label{eq:EquivalentTerminalCost}
	\sum_{k = 0}^{N(\alpha^{*})}\Gamma_{\iota^{*}_{k}}\mathbf{1}_{\lbrace \tau^{*}_{k} < T \rbrace}\mathbf{1}_{\lbrace \tau^{*}_{k+1} = T \rbrace} = \Gamma_{\iota^{*}_{N(\alpha^{*})}} \hspace{1em}\mathsf{P}-\text{a.s.}
	\end{equation}
	From equations~\eqref{eq:Verification5} and \eqref{eq:EquivalentTerminalCost}, one arrives at the following representation for $Y^{i}_{t}$: $\mathsf{P}$-a.s.,
	\begin{equation}\label{eq:Verification6}
	Y^{i}_{t} = \mathsf{E}\left[\sum_{s = t}^{T-1}\Psi_{\mathbf{u}^{*}_{s}}(s) + \Gamma_{\iota^{*}_{N(\alpha^{*})}} - \sum_{n \ge 1}\gamma_{\iota^{*}_{n-1},\iota^{*}_{n}}\left(\tau^{*}_{n}\right)\mathbf{1}_{\lbrace \tau^{*}_{n} < T  \rbrace} \biggm \vert \mathcal{F}_{t}\right] = J(\alpha^{*};t,i)
	\end{equation}
	
	Now let $\alpha = \left(\tau_{n},\iota_{n}\right)_{n \ge 0} \in \mathcal{A}_{t,i}$ be any admissible control. The verification theorem can be completed by showing that $J(\alpha^{*};t,i) \ge J(\alpha;t,i)$ a.s. First, note that $J(\alpha^{*};T,i) = J(\alpha;T,i)$ when $t = T$, so assume henceforth that $t < T$. Then, due to possible sub-optimality of the pair $(\tau_{1},\iota_{1})$, it is true that
	\begin{align*}
	Y^{i}_{t} & = \esssup_{\tau \in \mathcal{T}_{t}}\mathsf{E}\left[\sum_{s = t}^{\tau-1}\Psi_{i}(s) + U^{i}_{\tau} \biggm\vert \mathcal{F}_{t}\right] \\
	& \ge \mathsf{E}\left[\sum_{s = t}^{\tau_{1}-1}\Psi_{i}(s) + U^{i}_{\tau_{1}} \biggm\vert \mathcal{F}_{t}\right] \nonumber \\
	& = \mathsf{E}\left[\sum_{s = t}^{\tau_{1}-1}\Psi_{i}(s) + \Gamma_{i}\mathbf{1}_{\lbrace \tau_{1} = T \rbrace} + \max\limits_{j \neq i} \left\lbrace Y^{j}_{\tau_{1}} - \gamma_{i,j}\left(\tau_{1}\right) \right\rbrace \mathbf{1}_{\lbrace \tau_{1} < T \rbrace} \biggm\vert \mathcal{F}_{t}\right] \nonumber \\
	& \ge \mathsf{E}\left[\sum_{s = t}^{\tau_{1}-1}\Psi_{i}(s) + \Gamma_{i}\mathbf{1}_{\lbrace \tau_{1} = T \rbrace} + \left\lbrace Y^{\iota_{1}}_{\tau_{1}} - \gamma_{i,\iota_{1}}\left(\tau_{1}\right) \right\rbrace \mathbf{1}_{\lbrace \tau_{1} < T \rbrace} \biggm\vert \mathcal{F}_{t}\right].
	\end{align*}
	Repeating the arguments leading to \eqref{eq:Verification6}, replacing the equalities $(=)$ with inequalities ($\ge$) due to possible sub-optimality of $(\tau_{n},\iota_{n})$ for $n \ge 2$, eventually leads to
	\begin{align*}
	Y^{i}_{t} \ge {} & \mathsf{E}\left[\sum_{s = t}^{T-1}\Psi_{\mathbf{u}_{s}}(s) + \sum_{k = 0}^{N(\alpha) }\Gamma_{\iota_{k}}\mathbf{1}_{\lbrace \tau_{k} < T \rbrace}\mathbf{1}_{\lbrace \tau_{k+1} = T \rbrace} - \sum_{n = 1}^{N(\alpha) + 1}\gamma_{\iota_{n-1},\iota_{n}}\left(\tau_{n}\right)\mathbf{1}_{\lbrace \tau_{n} < T \rbrace}\biggm \vert \mathcal{F}_{t}\right] \\ 
	& + \mathsf{E}\left[Y^{\iota_{N(\alpha) + 1}}_{\tau_{N(\alpha) + 1}} \mathbf{1}_{\lbrace \tau_{N(\alpha) + 1} < T \rbrace} \biggm \vert \mathcal{F}_{t}\right] \\
	= {} & \mathsf{E}\left[\sum_{s = t}^{T-1}\Psi_{\mathbf{u}_{s}}(s) + \Gamma_{\iota_{N(\alpha)}} - \sum_{n \ge 1}\gamma_{\iota_{n-1},\iota_{n}}\left(\tau_{n}\right)\mathbf{1}_{\lbrace \tau_{n} < T\rbrace} \biggm \vert \mathcal{F}_{t}\right] \\
	= {} & J(\alpha;t,i)
	\end{align*}
	and proves that the strategy $\alpha^{*}$ is optimal.
\end{proof}
\section{Existence of the optimal processes}\label{Section:Discrete-Time-OS-Existence}
This section addresses the existence of the processes $Y^{1},\ldots,Y^{m}$ used in Theorem~\ref{theorem:verification}. The proof is a constructive one and verifies that the explicit dynamic programming scheme cited in \cite{Carmona2008,Gassiat2012} and elsewhere indeed solves the optimal switching problem.
\subsection{Backward dynamic programming}
\begin{lemma}[Backward Induction]\label{Lemma:BackwardInductionExistence}
	For each $i \in \mathbb{I}$, define the process $\tilde{Y}^{i} = \big(\tilde{Y}^{i}_{t}\big)_{t \in \mathbb{T}}$ recursively as follows:
	\begin{equation}\label{Definition:BackwardInduction}
	\begin{split}
	\tilde{Y}^{i}_{T} & = \Gamma_{i},\hspace{1 em}\text{ and for } t = T - 1,\ldots,0 : \\
	\tilde{Y}^{i}_{t} & = \max\limits_{j \neq i} \left\lbrace -\gamma_{i,j}\left(t\right) + \Psi_{j}(t) + \mathsf{E}\left[\tilde{Y}^{j}_{t+1} \vert \mathcal{F}_{t}\right] \right\rbrace \vee \left\lbrace \Psi_{i}(t) + \mathsf{E}\left[\tilde{Y}^{i}_{t+1} \vert \mathcal{F}_{t}\right]\right\rbrace.
	\end{split}
	\end{equation}
	Then $\tilde{Y}^{i}$ is $\mathbb{F}$-adapted and in $\mathcal{S}^{2}$.
\end{lemma}
\begin{proof}
	One verifies that $\tilde{Y}^{i}$ is a well-defined $\mathbb{F}$-adapted process by proceeding recursively for $t = T,\ldots,0$ using~\eqref{Definition:BackwardInduction}, noting that the conditional expectations are well-defined by the integrability conditions on the rewards and switching costs. Note that in this discrete-time setting, showing $\tilde{Y}^{i} \in \mathcal{S}^{2}$ is equivalent to showing $\tilde{Y}^{i}_{t} \in L^{2}$ for every $t \in \mathbb{T}$. Since $\tilde{Y}^{i}_{T} = \Gamma_{i} \in L^{2}$ for all $i \in \mathbb{I}$, the claim is true for $t = T$. Suppose by induction on $t = T-1,\ldots,0$ that $\tilde{Y}^{j}_{t+1} \in L^{2}$ for all $j \in \mathbb{I}$. The backward induction formula \eqref{Definition:BackwardInduction} gives:
	\begin{align}\label{eq:RecursiveIntegrability1}
	\left|\tilde{Y}^{i}_{t}\right| & = \left|\max\limits_{j \neq i} \left\lbrace -\gamma_{i,j}\left(t\right) + \Psi_{j}(t) + \mathsf{E}\left[\tilde{Y}^{j}_{t+1} \big\vert \mathcal{F}_{t}\right] \right\rbrace \vee \left\lbrace \Psi_{i}(t) + \mathsf{E}\left[\tilde{Y}^{i}_{t+1} \big\vert \mathcal{F}_{t}\right]\right\rbrace\right| \nonumber \\
	& \le \left|\max\limits_{j \neq i} \left\lbrace -\gamma_{i,j}\left(t\right) + \Psi_{j}(t) + \mathsf{E}\left[\tilde{Y}^{j}_{t+1} \big\vert \mathcal{F}_{t}\right] \right\rbrace\right| + \left|\left\lbrace \Psi_{i}(t) + \mathsf{E}\left[\tilde{Y}^{i}_{t+1} \big\vert \mathcal{F}_{t}\right]\right\rbrace\right| \nonumber \\
	& \le 2\max_{j \in I}\mathsf{E}\left[|\tilde{Y}^{j}_{t+1}|\big\vert \mathcal{F}_{t}\right] + \max\limits_{j,k \in \mathbb{I}}\max_{r \in \mathbb{T}}\left|\gamma_{j,k}(r)\right| + 2\max\limits_{j \in \mathbb{I}}\max_{r \in \mathbb{T}}\left|\Psi_{j}(r)\right|
	\end{align}
	Note that by Jensen's inequality (\cite[p.~139]{Rogers2000a}), the conditional expectation satisfies,
	\[
	\mathsf{E}\left[\left|\mathsf{E}\left[|\tilde{Y}^{j}_{t+1}|\big\vert \mathcal{F}_{t}\right]\right|^{2}\right] \le \mathsf{E}\left[|\tilde{Y}^{j}_{t+1}|^{2}\right] < +\infty,
	\]
	and is therefore also in $L^{2}$. In addition to the observation that $\mathbb{I}$ is finite and $\Psi_{j},\gamma_{j,k} \in \mathcal{S}^{2}$ for every $j,k \in \mathbb{I}$, the random variable on the right-hand side of \eqref{eq:RecursiveIntegrability1} is in $L^{2}$. Therefore $\tilde{Y}^{i}_{t} \in L^{2}$, which holds for every $i \in \mathbb{I}$ since $i$ was arbitrary. The case $t = T-1$ has already been verified so the proof by induction is complete.
\end{proof}
\subsubsection{An explicit Snell envelope system.}
	A connection between $\tilde{Y}^{i}$ in Lemma~\ref{Lemma:BackwardInductionExistence} and the Snell envelope becomes apparent upon defining a new process $\big(\hat{Y}^{i}_{t}\big)_{t \in \mathbb{T}}$ for every $i \in \mathbb{I}$ by,
	\begin{equation}\label{eq:ExistenceNewProcess}
	\hat{Y}^{i}_{t} \coloneqq \tilde{Y}^{i}_{t} + \sum_{s = 0}^{t-1} \Psi_{i}(s),
	\end{equation}
	A backward induction formula for $\hat{Y}^{i}_{t}$ is then obtained by adding the $\mathcal{F}_{t}$-measurable term $\sum_{s = 0}^{t-1} \Psi_{i}(s)$ to both sides of \eqref{Definition:BackwardInduction}:
	
	\begin{equation}\label{eq:ExistenceNewProcessBackwardInduction}
	\begin{split}
	\hat{Y}^{i}_{T} & = \hat{U}^{i}_{T},\hspace{1 em}\text{ and for } t = T - 1,\ldots,0 : \\
	\hat{Y}^{i}_{t} & = \hat{U}^{i}_{t} \vee \mathsf{E}\left[\hat{Y}^{i}_{t+1} \big\vert \mathcal{F}_{t}\right].
	\end{split}
	\end{equation}
	where $\big(\hat{U}^{i}_{t}\big)_{t \in \mathbb{T}}$ is the following \emph{explicit gain process}:
	\begin{align}\label{eq:ExistenceExplicitGainProcess}
	\hat{U}^{i}_{t} \coloneqq {} & \max\limits_{j \neq i} \left\lbrace -\gamma_{i,j}\left(t\right) + \sum_{s = 0}^{t-1} \left(\Psi_{i}(s) - \Psi_{j}(s)\right) + \mathsf{E}\left[\hat{Y}^{j}_{t+1} \vert \mathcal{F}_{t}\right] \right\rbrace \mathbf{1}_{\lbrace t < T \rbrace} \nonumber \\
	& + \left\lbrace \sum_{s = 0}^{T-1} \Psi_{i}(s) + \Gamma_{i} \right \rbrace\mathbf{1}_{\lbrace t = T \rbrace} \nonumber \\
	= {} & \sum_{s=0}^{t-1} \Psi_{i}(s) + \max\limits_{j \neq i} \left\lbrace -\gamma_{i,j}\left(t\right) - \sum_{s = 0}^{t-1} \Psi_{j}(s) + \mathsf{E}\left[\hat{Y}^{j}_{t+1} \vert \mathcal{F}_{t}\right] \right\rbrace \mathbf{1}_{\lbrace t < T \rbrace} \nonumber \\
	& + \Gamma_{i}\mathbf{1}_{\lbrace t = T \rbrace}.
	\end{align}
	The processes $\big(\hat{Y}^{i}_{t}\big)_{t \in \mathbb{T}}$ and $\big(\hat{U}^{i}_{t}\big)_{t \in \mathbb{T}}$ belong to $\mathcal{S}^{2}$ by integrability properties of the rewards, switching costs and as $\tilde{Y}^{i} \in \mathcal{S}^{2}$. Proposition~\ref{Proposition:DiscreteSnellEnvelopeProperties} verifies that the backward induction formula uniquely defines $\big(\hat{Y}^{i}_{t}\big)_{t \in \mathbb{T}}$ as the Snell envelope of $\big(\hat{U}^{i}_{t}\big)_{t \in \mathbb{T}}$.
\begin{theorem}[Existence]\label{theorem:ExistenceSnellEnevelope}
	Let $\big(\tilde{Y}^{i}_{t}\big)_{t \in \mathbb{T}}$, $i \in \mathbb{I}$, be the processes defined by backward induction \eqref{Definition:BackwardInduction}. Then, $\mathsf{P}$-a.s. for every $t \in \mathbb{T}$:
	
	\begin{equation}\label{eq:BackwardInductionSnellEnvelope}
	\tilde{Y}^{i}_{t} = \esssup\limits_{\tau \in \mathcal{T}_{t}} \mathsf{E}\left[\sum_{s=t}^{\tau-1}\Psi_{i}(s) + \Gamma_{i}\mathbf{1}_{\lbrace \tau = T \rbrace} + \max\limits_{j \neq i} \left\lbrace \tilde{Y}^{j}_{\tau} - \gamma_{i,j}\left(\tau\right) \right\rbrace \mathbf{1}_{\lbrace \tau < T \rbrace} \biggm \vert \mathcal{F}_{t}\right].
	\end{equation}
	Therefore, $\tilde{Y}^{1},\ldots,\tilde{Y}^{m}$ satisfy the verification theorem.
\end{theorem}
\begin{proof}
	For notational convenience, introduce a new process $\big(\hat{W}^{i}_{t}\big)_{t \in \mathbb{T}}$ which is defined for $t \in \mathbb{T}$ by
	\begin{equation}\label{eq:Discrete-Time-Optimal-Switching-Implicit-Gain}
	\hat{W}^{i}_{t} \coloneqq \sum_{s=0}^{t-1} \Psi_{i}(s) + \Gamma_{i}\mathbf{1}_{\lbrace t = T \rbrace} + \max\limits_{j \neq i} \left\lbrace -\gamma_{i,j}\left(t\right) - \sum_{s = 0}^{t-1} \Psi_{j}(s) + \hat{Y}^{j}_{t} \right\rbrace \mathbf{1}_{\lbrace t < T \rbrace}. 
	\end{equation}
	Note that $\hat{W}^{i} \in \mathcal{S}^{2}$ by the properties of $\Gamma_{i},\Psi_{i},\gamma_{i,j}$ and $\hat{Y}^{j}$ for $i,j \in \mathbb{I}$. Equation~\eqref{eq:BackwardInductionSnellEnvelope} can be proved if 
	\begin{equation}\label{eq:BackwardInductionSnellEnvelopeNewProcess}
	\forall t \in \mathbb{T}: \hat{Y}^{i}_{t} = \esssup\limits_{\tau \in \mathcal{T}_{t}} \mathsf{E}\bigl[\hat{W}^{i}_{\tau} \big \vert \mathcal{F}_{t}\bigr], \enskip \mathsf{P}-\text{a.s.}
	\end{equation}
	
	Indeed, if equation~\eqref{eq:BackwardInductionSnellEnvelopeNewProcess} is true then by~\eqref{eq:ExistenceNewProcess}: $\mathsf{P}-a.s.$,
	\begin{align*}
	\tilde{Y}^{i}_{t} & = \esssup\limits_{\tau \in \mathcal{T}_{t}} \mathsf{E}\bigl[\hat{W}^{i}_{\tau} \big \vert \mathcal{F}_{t}\bigr] - \sum_{s = 0}^{t-1} \Psi_{i}(s) \\
	& = \esssup\limits_{\tau \in \mathcal{T}_{t}} \mathsf{E}\left[\sum_{s = 0}^{\tau-1} \Psi_{i}(s) + \Gamma_{i}\mathbf{1}_{\lbrace \tau = T \rbrace} + \max\limits_{j \neq i} \left\lbrace \tilde{Y}^{j}_{\tau} - \gamma_{i,j}\left(\tau\right) \right\rbrace \mathbf{1}_{\lbrace \tau < T \rbrace} \biggm \vert \mathcal{F}_{t}\right] - \sum_{s = 0}^{t-1} \Psi_{i}(s) \\
	& = \esssup\limits_{\tau \in \mathcal{T}_{t}} \mathsf{E}\left[\sum_{s = t}^{\tau-1} \Psi_{i}(s) + \Gamma_{i}\mathbf{1}_{\lbrace \tau = T \rbrace} + \max\limits_{j \neq i} \left\lbrace \tilde{Y}^{j}_{\tau} - \gamma_{i,j}\left(\tau\right) \right\rbrace \mathbf{1}_{\lbrace \tau < T \rbrace} \biggm \vert \mathcal{F}_{t}\right].
	\end{align*}
	
	In order to prove~\eqref{eq:BackwardInductionSnellEnvelopeNewProcess}, first note by Proposition~\ref{Proposition:DiscreteSnellEnvelopeProperties} the Snell envelope of $\hat{W}^{i}$ exists and, denoting it by $Z^{i} = (Z^{i}_{t})_{t \in \mathbb{T}}$, satisfies
	\[
	Z^{i}_{t} = \esssup\limits_{\tau \in \mathcal{T}_{t}} \mathsf{E}\bigl[\hat{W}^{i}_{\tau} \big \vert \mathcal{F}_{t}\bigr]
	\]
	as well as the backward induction formula
	\begin{equation}\label{eq:Discrete-Time-Optimal-Switching-Backward-Induction-2}
	\begin{split}
	Z^{i}_{T} & = \hat{W}^{i}_{T} ,\hspace{1 em}\text{ and for } t = T - 1,\ldots,0 : \\
	Z^{i}_{t} & = \hat{W}^{i}_{t} \vee \mathsf{E}\bigl[Z^{i}_{t+1} \big\vert \mathcal{F}_{t}\bigr].
	\end{split}
	\end{equation}
	Thus establishing \eqref{eq:BackwardInductionSnellEnvelopeNewProcess} is equivalent to showing that $\hat{Y}^{i}$ is a modification of $Z^{i}$ defined in \eqref{eq:Discrete-Time-Optimal-Switching-Backward-Induction-2}. Note that by Proposition \MakeUppercase{\romannumeral 2}.36.5 of \cite{Rogers2000a} this would also mean that $\hat{Y}^{i}$ and $Z^{i}$ are indistinguishable. This shall be proved inductively.
	
	Note that $Z^{i}_{T} = \hat{W}^{i}_{T} = \hat{U}^{i}_{T} = \hat{Y}^{i}_{T}$ almost surely for every $i \in \mathbb{I}$. Suppose inductively for $t = T-1,\ldots,0$ that for all $i \in \mathbb{I}$, $Z^{i}_{t+1} = \hat{Y}^{i}_{t+1}$ $\mathsf{P}$-a.s. For every $i \in \mathbb{I}$ define the stopping time $\theta^{i}_{t}$ as follows
	\begin{align}\label{Definition:BackwardInductionStoppingTimes}
	\theta^{i}_{t} = \inf\left\lbrace t \le s \le T \colon \hat{Y}^{i}_{s} = \hat{U}^{i}_{s} \right\rbrace
	\end{align}
	noting that $t \le \theta^{i}_{t} \le T$ almost surely. The following lines will establish $Z^{i}_{t} = \hat{Y}^{i}_{t}$ separately on the events $\{\theta^{i}_{t} = t\}$ and $\{\theta^{i}_{t} > t\} \equiv \{\theta^{i}_{t} \ge t + 1\}$. Since $\mathsf{P}(\{\theta^{i}_{t} \ge t\}) = 1$, the previous claim would lead to $Z^{i}_{t} = \hat{Y}^{i}_{t}$ a.s. and the induction argument proves that $\hat{Y}^{i}$ is a modification of $Z^{i}$.
	
	\paragraph{Case 1: $Z^{i}_{t} = \hat{Y}^{i}_{t} ~on~ \{\theta^{i}_{t} = t\}$.\\}
	Since $\hat{Y}^{j}$ is the Snell envelope of $\hat{U}^{j}$ for $j \in \mathbb{I}$, it is a supermartingale and by definition of $\hat{U}^{i}$ and $\hat{W}^{i}$ this leads to $\hat{U}^{i} \le \hat{W}^{i} \le Z^{i}$ for $i \in \mathbb{I}$. Then, using \eqref{Definition:BackwardInductionStoppingTimes}, the backward induction formula and the induction hypothesis one gets
	\begin{equation}\label{eq:Proof-BackwardInductionSnellEnvelopeNewProcess-1}
	\hat{W}^{i}_{t} \ge \hat{U}^{i}_{t} = \hat{Y}^{i}_{t} \ge \mathsf{E}\bigl[\hat{Y}^{i}_{t+1} \big\vert \mathcal{F}_{t}\bigr] = \mathsf{E}\bigl[Z^{i}_{t+1} \vert \mathcal{F}_{t}\bigr] \enskip \text{on} \enskip \{\theta^{i}_{t} = t\}.
	\end{equation}
	
	Using the backward induction formula \eqref{eq:Discrete-Time-Optimal-Switching-Backward-Induction-2} for $Z^{i}$ and \eqref{eq:Proof-BackwardInductionSnellEnvelopeNewProcess-1} above also shows that
	\begin{equation}\label{eq:Proof-BackwardInductionSnellEnvelopeNewProcess-2}
	\hat{W}^{i}_{t} = Z^{i}_{t} \enskip \text{on} \enskip \{\theta^{i}_{t} = t\}
	\end{equation}
	Using \eqref{eq:Proof-BackwardInductionSnellEnvelopeNewProcess-2} and finiteness of $\mathbb{I}$ shows that there exists an $\mathcal{F}_{\theta^{i}_{t}}$-measurable mode $j_{*}$ (that is, $j_{*}$ is an $\mathcal{F}_{\theta^{i}_{t}}$-measurable $\mathbb{I}$-valued random variable) such that
	\begin{equation}\label{eq:Proof-BackwardInductionSnellEnvelopeNewProcess-3}
	\begin{split}
	Z^{i}_{t} & = \sum_{s=0}^{t-1} \Psi_{i}(s) - \gamma_{i,j_{*}}(t) - \sum_{s = 0}^{t-1} \Psi_{j_{*}}(s) + \hat{Y}^{j_{*}}_{t}, \\
	j_{*} & = \argmax\limits_{j \neq i} \left\lbrace -\gamma_{i,j}\left(t\right) - \sum_{s = 0}^{t-1} \Psi_{j}(s) + \hat{Y}^{j}_{t} \right\rbrace \enskip \text{on} \enskip \{\theta^{i}_{t} = t\}.
	\end{split}
	\end{equation}
	Let us show that
	\begin{equation}\label{eq:Proof-BackwardInductionSnellEnvelopeNewProcess-4}
	\mathbf{1}_{\{\theta^{i}_{t} = t\}}\hat{Y}^{j_{*}}_{t} = \mathbf{1}_{\{\theta^{i}_{t} = t\}}\mathsf{E}\bigl[\hat{Y}^{j_{*}}_{t+1} \big\vert \mathcal{F}_{t}\bigr] \quad \mathsf{P}-\text{a.s.}
	\end{equation}
	Since $j_{*}$ is $\mathcal{F}_{\theta^{i}_{t}}$-measurable, one has for $t \le r \le T$
	\[
	\mathsf{E}\left[\sum_{j \in \mathbb{I}}\mathbf{1}_{\{j_{*} = j\}}\hat{Y}^{j}_{r} \biggm\vert \mathcal{F}_{t}\right] = \sum_{j \in \mathbb{I}}\mathbf{1}_{\{j_{*} = j\}}\mathsf{E}\bigl[\hat{Y}^{j}_{r} \big\vert \mathcal{F}_{t}\bigr] \le \sum_{j \in \mathbb{I}}\mathbf{1}_{\{j_{*} = j\}}\hat{Y}^{j}_{t} \enskip \text{ on } \enskip \{\theta^{i}_{t} = t \} 
	\]
	so that $\hat{Y}^{j_{*}} \coloneqq \sum_{j \in \mathbb{I}}\hat{Y}^{j}\mathbf{1}_{\{j_{*} = j\}}$ is a supermartingale on $[\theta^{i}_{t},T]$. Now if \eqref{eq:Proof-BackwardInductionSnellEnvelopeNewProcess-4} is not true, by the supermartingale property of $\hat{Y}^{j_{*}}$ the event $A^{i}_{t}$ defined by
	\[
	A^{i}_{t} \coloneqq \big\{\hat{Y}^{j_{*}}_{t} > \mathsf{E}\bigl[\hat{Y}^{j_{*}}_{t+1} \big \vert \mathcal{F}_{t}\bigr]\big\} \cap \{\theta^{i}_{t} = t\}
	\]
	has positive probability. If $\mathsf{P}(A^{i}_{t}) > 0$, there exists an $\mathcal{F}_{\theta^{i}_{t}}$-measurable mode $k_{*}$ such that
	\[
	\hat{Y}^{j_{*}}_{t} = \sum_{s=0}^{t-1} \Psi_{j_{*}}(s) -\gamma_{j_{*},k_{*}}(t) - \sum_{s = 0}^{t-1} \Psi_{k_{*}}(s) + \mathsf{E}\bigl[\hat{Y}^{k_{*}}_{t+1} \vert \mathcal{F}_{t} \bigr] \enskip \text{on} \enskip A^{i}_{t}
	\]
	This leads to
	\begin{align}
	Z^{i}_{t} = {} & \sum_{s=0}^{t-1} \Psi_{i}(s) -\gamma_{i,j_{*}}(t) - \sum_{s = 0}^{t-1} \Psi_{j_{*}}(s) + \hat{Y}^{j_{*}}_{t}  \nonumber \\
	= {} & \sum_{s=0}^{t-1} \Psi_{i}(s) -\gamma_{i,j_{*}}(t) - \sum_{s = 0}^{t-1} \Psi_{j_{*}}(s) + \sum_{s=0}^{t-1} \Psi_{j_{*}}(s) \nonumber \\
	& -\gamma_{j_{*},k_{*}}(t) - \sum_{s = 0}^{t-1} \Psi_{k_{*}}(s) + \mathsf{E}\bigl[\hat{Y}^{k_{*}}_{t+1} \big\vert \mathcal{F}_{t} \bigr] \nonumber \\
	& < \sum_{s=0}^{t-1} \Psi_{i}(s) -\gamma_{i,k_{*}}(t) - \sum_{s = 0}^{t-1} \Psi_{k_{*}}(s) + \hat{Y}^{k_{*}}_{t} \label{eq:Proof-BackwardInductionSnellEnvelopeNewProcess-5}  \enskip \text{on} \enskip A^{i}_{t}
	\end{align}
	where the inequality comes from the no-arbitrage condition~\eqref{assumption:Discrete-Time-SwitchingCosts-No-Arbitrage} and the supermartingale property of $\hat{Y}^{k_{*}}$ on $[\theta^{i}_{t},T]$. However, this contradicts the optimality of $j_{*}$ and therefore shows that \eqref{eq:Proof-BackwardInductionSnellEnvelopeNewProcess-4} holds.
	
	Using \eqref{eq:Proof-BackwardInductionSnellEnvelopeNewProcess-3} together with \eqref{eq:Proof-BackwardInductionSnellEnvelopeNewProcess-4} and the definition of $\hat{Y}^{i}$ yields:
	\[
	Z^{i}_{t} = \sum_{s=0}^{t-1} \Psi_{i}(s) -\gamma_{i,j_{*}}(t) - \sum_{s = 0}^{t-1} \Psi_{j_{*}}(s) + \mathsf{E}\bigl[\hat{Y}^{j_{*}}_{t+1} \big\vert \mathcal{F}_{t}\bigr] \le \hat{Y}^{i}_{t} \enskip \text{on} \enskip \{\theta^{i}_{t} = t\}
	\]
	and using this with \eqref{eq:Proof-BackwardInductionSnellEnvelopeNewProcess-1} and \eqref {eq:Proof-BackwardInductionSnellEnvelopeNewProcess-2} gives:
	\begin{equation}\label{eq:Proof-BackwardInductionSnellEnvelopeNewProcess-6}
	\hat{W}^{i}_{t} = \hat{U}^{i}_{t} = \hat{Y}^{i}_{t} = Z^{i}_{t} \enskip \text{on} \enskip \{\theta^{i}_{t} = t\}.
	\end{equation}
	
	\paragraph{Case 2: $Z^{i}_{t} = \hat{Y}^{i}_{t} ~on~ \{\theta^{i}_{t} \ge t + 1\}$.\\}
	Note that $\{\theta^{i}_{t} \ge t + 1\} \equiv \{\theta^{i}_{t} > t\}$ and is therefore $\mathcal{F}_{t}$-measurable. The properties of the Snell envelopes $\hat{Y}^{i}$ and $Z^{i}$ together with the induction hypothesis then give
	\begin{equation}\label{eq:Proof-BackwardInductionSnellEnvelopeNewProcess-7}
	\hat{Y}^{i}_{t} = \mathsf{E}\bigl[\hat{Y}^{i}_{t+1} \big \vert \mathcal{F}_{t}\bigr] = \mathsf{E}\bigl[Z^{i}_{t+1} \big \vert \mathcal{F}_{t}\bigr] \le Z^{i}_{t} \enskip \text{on} \enskip \{\theta^{i}_{t} \ge t + 1\}.
	\end{equation}
	Suppose that $Z^{i}$ is a strict supermartingale on $\{\theta^{i}_{t} \ge t + 1\}$ with positive probability. Then the $\mathcal{F}_{t}$-measurable event $B^{i}_{t}$ defined by
	\[
	B^{i}_{t} \coloneqq \left\{\mathsf{E}\left[Z^{i}_{t+1} \big \vert \mathcal{F}_{t}\right] < Z^{i}_{t} \right\} \cap \{\theta^{i}_{t} \ge t + 1\}
	\]
	has positive probability, and implies the existence of an $\mathcal{F}_{t}$-measurable mode $j_{*}$ such that
	\begin{equation}\label{eq:Proof-BackwardInductionSnellEnvelopeNewProcess-8}
	\begin{split}
	\hat{Y}^{i}_{t} < Z^{i}_{t} & = \sum_{s=0}^{t-1} \Psi_{i}(s) - \gamma_{i,j_{*}}(t) - \sum_{s = 0}^{t-1} \Psi_{j_{*}}(s) + \hat{Y}^{j_{*}}_{t}, \\
	j_{*} & = \argmax\limits_{j \neq i} \left\lbrace -\gamma_{i,j}\left(t\right) - \sum_{s = 0}^{t-1} \Psi_{j}(s) + \hat{Y}^{j}_{t} \right\rbrace \enskip \text{on} \enskip B^{i}_{t}.
	\end{split}
	\end{equation}
	But by definition of $\hat{Y}^{i}$, it is true that
	\[
	\hat{Y}^{i}_{t} \ge \sum_{s=0}^{t-1} \Psi_{i}(s) - \gamma_{i,j_{*}}(t) - \sum_{s = 0}^{t-1} \Psi_{j_{*}}(s) + \mathsf{E}\bigl[\hat{Y}^{j_{*}}_{t+1} \big\vert \mathcal{F}_{t}\bigr] \enskip \text{on} \enskip B^{i}_{t}
	\]
	and using this in \eqref{eq:Proof-BackwardInductionSnellEnvelopeNewProcess-8} shows
	\[
	\mathsf{E}\bigl[\hat{Y}^{j_{*}}_{t+1} \big\vert \mathcal{F}_{t}\bigr] < \hat{Y}^{j_{*}}_{t} \enskip \text{on} \enskip B^{i}_{t}.
	\]
	Using the same arguments leading up to \eqref{eq:Proof-BackwardInductionSnellEnvelopeNewProcess-5}, one can show that there exists an $\mathcal{F}_{t}$-measurable mode $k_{*}$ such that
	\begin{align*}
	Z^{i}_{t} = {} & \sum_{s=0}^{t-1} \Psi_{i}(s) -\gamma_{i,j_{*}}(t) - \sum_{s = 0}^{t-1} \Psi_{j_{*}}(s) + \hat{Y}^{j_{*}}_{t}  \\
	= {} & \sum_{s=0}^{t-1} \Psi_{i}(s) -\gamma_{i,j_{*}}(t) - \sum_{s = 0}^{t-1} \Psi_{j_{*}}(s) + \sum_{s=0}^{t-1} \Psi_{j_{*}}(s) \nonumber \\
	& -\gamma_{j_{*},k_{*}}(t) - \sum_{s = 0}^{t-1} \Psi_{k_{*}}(s) + \mathsf{E}\bigl[\hat{Y}^{k_{*}}_{t+1} \vert \mathcal{F}_{t} \bigr] \nonumber \\
	& < \sum_{s=0}^{t-1} \Psi_{i}(s) -\gamma_{i,k_{*}}(t) - \sum_{s = 0}^{t-1} \Psi_{k_{*}}(s) + \hat{Y}^{k_{*}}_{t}  \enskip \text{on} \enskip B^{i}_{t}
	\end{align*}
	which contradicts the optimality of $j_{*}$. Thus $\mathsf{P}\big(B^{i}_{t}\big) = 0$ and this shows
	\begin{equation}\label{eq:Proof-BackwardInductionSnellEnvelopeNewProcess-9}
	\mathbf{1}_{\{\theta^{i}_{t} \ge t + 1\}}\mathsf{E}\bigl[Z^{i}_{t+1} \big \vert \mathcal{F}_{t}\bigr] = Z^{i}_{t}\mathbf{1}_{\{\theta^{i}_{t} \ge t + 1\}},\enskip \mathsf{P}-\text{a.s.}
	\end{equation}
	Finally, putting \eqref{eq:Proof-BackwardInductionSnellEnvelopeNewProcess-9}, \eqref{eq:Proof-BackwardInductionSnellEnvelopeNewProcess-7} and \eqref{eq:Proof-BackwardInductionSnellEnvelopeNewProcess-6} together gives
	\[
	\hat{Y}^{i}_{t} = Z^{i}_{t},\quad \mathsf{P}-\text{a.s.}
	\]
	Since the case $t = T-1$ is true and $i \in \mathbb{I}$ was arbitrary, the proof by induction is complete. Therefore, for every $i \in \mathbb{I}$,
	\[
	\forall t \in \mathbb{T}: \hat{Y}^{i}_{t} = Z^{i}_{t} \quad \mathsf{P}-\text{a.s.}
	\]
	which means $\hat{Y}^{i}$ is a modification of (and therefore indistinguishable from) $Z^{i}$, whence~\eqref{eq:BackwardInductionSnellEnvelopeNewProcess} follows.
\end{proof}
\begin{remark}
	By equation~\eqref{eq:BackwardInductionSnellEnvelopeNewProcess} and Proposition~\ref{Proposition:DiscreteSnellEnvelopeProperties}, the arguments used in the previous proof show that $\hat{Y}^{i},~i \in \mathbb{I}$, satisfies:
	\[
	\begin{split}
	\tilde{Y}^{i}_{T} & = \Gamma_{i},\hspace{1 em}\text{ and for } t = T - 1,\ldots,0 : \\
	\tilde{Y}^{i}_{t} & = \max\limits_{j \neq i} \left\lbrace -\gamma_{i,j}\left(t\right) + \tilde{Y}^{j}_{t} \right\rbrace \vee \left\lbrace \Psi_{i}(t) + \mathsf{E}\bigl[\tilde{Y}^{i}_{t+1} \big\vert \mathcal{F}_{t}\bigr]\right\rbrace.
	\end{split}
	\]
	which is the (non-Markovian) analogue to the implicit backward induction formula of \cite{Gassiat2012}. However, the steps carried out here are in the reverse direction to \cite[pp.~2037]{Gassiat2012}, where the explicit backward induction formula was derived from the implicit one. The latter was obtained from an approximation to a continuous-time optimal switching problem with non-negative switching costs.
\end{remark}

\section{Conclusion}
In this paper we used a probabilistic approach to solve a finite-horizon discrete-time optimal switching problem for a model with a running reward, terminal reward and signed switching costs. The approach, which works without Markovian assumptions, reduced the switching problem to iterated optimal stopping problems defined in terms of (coupled) Snell envelopes, just as in the verification theorem of \cite{Djehiche2009} in the continuous-time case. We were able to define the Snell envelopes by an explicit backward induction scheme, thereby extending the numerical methods of \cite{Carmona2008,Gassiat2012} to problems with negative switching costs.

\appendix
\numberwithin{equation}{section}
\section{Supplementary proofs}\label{Section:Appendix}
\begin{lemma}\label{lemma:VerificationLemma}
	For each $i \in \mathbb{I}$, let $U^{i}  \in \mathcal{S}^{2}$ and $Y^{i} \in \mathcal{S}^{2}$ be defined as in equations~\eqref{eq:VerificationImplicitGainProcess} and \eqref{eq:VerificationSnellEnvelope2} respectively. Let $\tau_{n} \in \mathcal{T}$ and $\iota_{n} \colon \Omega \to \mathbb{I}$ be $\mathcal{F}_{\tau_{n}}$-measurable. Then,
	\begin{equation}\label{eq:VerificationLemmaClaim} 
	Y^{\iota_{n}}_{t} = \esssup\limits_{\tau \in \mathcal{T}_{t}}\mathsf{E}\left[\sum_{s = t}^{\tau-1} \Psi_{\iota_{n}}(s) + U^{\iota_{n}}_{\tau} \biggm \vert \mathcal{F}_{t}\right] \enskip \text{on} \enskip [\tau_{1},T].
	\end{equation}
\end{lemma}
\begin{proof}
	For notational simplicity define $(\check{U}^{i}_{t})_{t \in \mathbb{T}}$ by $\check{U}^{i}_{t} = \sum_{s = 0}^{t-1} \Psi_{i}(s) + U^{i}_{t}$. For any $i \in \mathbb{I}$ and any time $s \le t$, $\Psi_{i}(s)$ is $\mathcal{F}_{t}$-measurable and using this in equation~\eqref{eq:VerificationSnellEnvelope2} shows,
	\begin{align}
	Y^{i}_{t} & = \esssup\limits_{\tau \in \mathcal{T}_{t}}\mathsf{E}\left[\sum_{s = t}^{\tau-1} \Psi_{i}(s)  + U^{i}_{\tau} \biggm \vert \mathcal{F}_{t}\right] \nonumber\\
	& = \esssup\limits_{\tau \in \mathcal{T}_{t}}\mathsf{E}\left[\sum_{s = 0}^{\tau-1} \Psi_{i}(s) - \sum_{s = 0}^{t-1} \Psi_{i}(s) + U^{i}_{\tau} \biggm \vert \mathcal{F}_{t}\right] \nonumber \\
	& = -\sum_{s = 0}^{t-1}\Psi_{i}(s) + \esssup\limits_{\tau \in \mathcal{T}_{t}}\mathsf{E}\left[\check{U}^{i}_{\tau} \big \vert \mathcal{F}_{t}\right]\label{eq:VerificationLemma1}.
	\end{align}
	
	Since $U^{i},\Psi_{i} \in \mathcal{S}^{2}$, the Snell envelope of the process $\left(\sum_{s = 0}^{t-1} \Psi_{i}(s) + U^{i}_{t}\right)_{t \in \mathbb{T}}$ exists (cf. Proposition~\ref{Proposition:DiscreteSnellEnvelopeProperties}) and is denoted by $\check{Y}^{i}$. Furthermore, using equation~\eqref{eq:VerificationLemma1}, $\check{Y}^{i}$ satisfies
	\begin{equation}\label{eq:VerificationLemmaSnellEnvelopeDiscreteTime}
	\check{Y}^{i}_{t} = \esssup\limits_{\tau \in \mathcal{T}_{t}}\mathsf{E}\left[\check{U}^{i}_{\tau} \big \vert \mathcal{F}_{t}\right] = Y^{i}_{t} + \sum_{s = 0}^{t-1}\Psi_{i}(s)
	\end{equation}
	
	In particular $\check{Y}^{i}$ is the smallest supermartingale which dominates $\check{U}^{i}$. Note that as $Y^{i},\Psi_{i} \in \mathcal{S}^{2}$, the supermartingale property carries over to stopping times by Doob's Optional Sampling Theorem (Theorem~\MakeUppercase{\romannumeral 2}.59.1 of \cite{Rogers2000a}).
	
	Consider the process $\sum_{i \in \mathbb{I}}\mathbf{1}_{\lbrace \iota_{n} = i \rbrace} \check{Y}^{i}$ on $[\tau_{n},T]$ and remember that the sum over $\mathbb{I}$ is finite. Let $r,t \in \mathbb{T}$ be arbitrary times satisfying $r \le t$. Note that the indicator function $\mathbf{1}_{\lbrace \iota_{n} = i \rbrace}$ is non-negative, and each $\mathbf{1}_{\lbrace \iota_{n} = i \rbrace}$ is $\mathcal{F}_{\tau_{n}}$-measurable and therefore $\mathcal{F}_{r}$-measurable on $\{\tau_{n} \le r\}$. Using these observations together with the supermartingale property yields: almost surely,
	\[
	\mathsf{E}\left[\sum\limits_{i \in \mathbb{I}}\mathbf{1}_{\lbrace \iota_{n} = i \rbrace}\check{Y}^{i}_{t} \biggm \vert \mathcal{F}_{r}\right]\mathbf{1}_{\{\tau_{n} \le r\}} = \sum\limits_{i \in \mathbb{I}}\mathbf{1}_{\lbrace \iota_{n} = i \rbrace}\mathsf{E}\left[\check{Y}^{i}_{t} \big \vert \mathcal{F}_{r}\right]\mathbf{1}_{\{\tau_{n} \le r \}} \le \sum\limits_{i \in \mathbb{I}}\mathbf{1}_{\lbrace \iota_{n} = i \rbrace}\mathbf{1}_{\{\tau_{n} \le r \}}\check{Y}^{i}_{r}
	\]
	
	This shows $\sum_{i \in \mathbb{I}}\mathbf{1}_{\lbrace \iota_{n} = i \rbrace}\check{Y}^{i}$ is a supermartingale on $[\tau_{n},T]$. For each $i \in \mathbb{I}$, the dominating property of the Snell envelope and non-negativity of $\mathbf{1}_{\lbrace \iota_{n} = i \rbrace}$ leads to: 
	\[
	\mathbf{1}_{\{\tau_{n} \le t\}}\mathbf{1}_{\lbrace \iota_{n} = i \rbrace}\check{Y}^{i}_{t} \ge \mathbf{1}_{\{\tau_{n} \le t\}}\mathbf{1}_{\lbrace \iota_{n} = i \rbrace}\check{U}^{i}_{t}
	\]
	and summing over $i \in \mathbb{I}$ then gives,
	\begin{align*}
	\check{Y}^{\iota_{n}}_{t} \coloneqq \sum\limits_{i \in \mathbb{I}}\mathbf{1}_{\lbrace \iota_{n} = i \rbrace}\check{Y}^{i}_{t} \ge \sum\limits_{i \in \mathbb{I}}\mathbf{1}_{\lbrace \iota_{n} = i \rbrace}\check{U}^{i}_{t} \eqqcolon \check{U}^{\iota_{n}}_{t} \quad \text{ on } \{\tau_{n} \le t\}.
	\end{align*}
	
	The process $\check{Y}^{\iota_{n}}$ is therefore a supermartingale dominating $\check{U}^{\iota_{n}}$ on $[\tau_{n},T]$. Similar arguments as above can be used to show that $\check{Y}^{\iota_{n}}$ is the smallest supermartingale with this property, and is therefore the Snell envelope of $\check{U}^{\iota_{n}}$. Proposition~\ref{Proposition:DiscreteSnellEnvelopeProperties} leads to a representation for $\check{Y}^{\iota_{n}}$ similar to \eqref{eq:VerificationLemmaSnellEnvelopeDiscreteTime}, and the $\mathcal{F}_{t}$-measurability of the summation term leads to equation~\eqref{eq:VerificationLemmaClaim}.
\end{proof}
\begin{lemma}\label{lemma:AdmissibleOptimalControl}
	Let $\alpha^{*} = \left(\tau^{*}_{n},\iota^{*}_{n}\right)_{n \ge 0}$ be the sequence given in equation~\eqref{eq:OptimalStoppingStrategy}. Suppose that Assumption~\ref{assumption:SwitchingCosts} holds for the switching costs. Then $\alpha^{*} \in \mathcal{A}_{t,i}$.
\end{lemma}
\begin{proof}
	The times $\left\lbrace\tau^{*}_{n}\right\rbrace_{n \ge 0}$ are non-decreasing by definition, $\tau^{*}_{0} = t$ and each $\tau^{*}_{n} \in \mathcal{T}_{t}$ since $U^{i}$ and $Y^{i}$ are adapted for every $i \in \mathbb{I}$. Corollary \MakeUppercase{\romannumeral 2}-1-4 of \cite{Neveu1975} states that for any adapted process $Z$ and stopping time $\tau$, $Z_{\tau}$ is $\mathcal{F}_{\tau}$-measurable. The sets $A^{\iota^{*}_{n-1}}_{j}$ in equation~\eqref{eq:OptimalStoppingStrategy} are therefore $\mathcal{F}_{\tau^{*}_{n}}$-measurable sets which means the modes $\left\lbrace\iota^{*}_{n}\right\rbrace_{n \ge 0}$ are also $\mathcal{F}_{\tau^{*}_{n}}$-measurable. Furthermore, $\iota^{*}_{n} \neq \iota^{*}_{n+1}$ almost surely for $n \ge 0$.
	
	The last thing to verify is $\mathsf{P}\left(\{\tau^{*}_{n} < T \} \cap \{\tau^{*}_{n} = \tau^{*}_{n+1}\}\right) = 0$ for $n \ge 1$. Assume contrarily that for some $n \ge 1$, the event $\{\tau^{*}_{n} < T \} \cap \{\tau^{*}_{n} = \tau^{*}_{n+1}\}$ has positive probability (recall $\tau^{*}_{n+1} \ge \tau^{*}_{n}$). By equation~\eqref{eq:OptimalStoppingStrategy} for $\tau^{*}_{n}$ and $\tau^{*}_{n+1}$, we have $\mathsf{P}$-almost surely:
	\[
	Y^{\iota^{*}_{n-1}}_{\tau^{*}_{n}} = U^{\iota^{*}_{n-1}}_{\tau^{*}_{n}},\hspace{1em} Y^{\iota^{*}_{n}}_{\tau^{*}_{n+1}} = U^{\iota^{*}_{n}}_{\tau^{*}_{n+1}}.
	\]
	
	By definition of $\iota^{*}_{n}$ and $\iota^{*}_{n+1}$, on the event $\lbrace \tau^{*}_{n} < T \rbrace \cap \lbrace \tau^{*}_{n} = \tau^{*}_{n+1} \rbrace$ it is also true that:
	\begin{equation}\label{eq:AdmissibilityLemmaContradiction}
	Y^{\iota^{*}_{n-1}}_{\tau^{*}_{n}} = -\gamma_{\iota^{*}_{n-1},\iota^{*}_{n}}\left(\tau^{*}_{n}\right) + Y^{\iota^{*}_{n}}_{\tau^{*}_{n}} \quad \text{and} \quad
	Y^{\iota^{*}_{n}}_{\tau^{*}_{n}} = -\gamma_{\iota^{*}_{n},\iota^{*}_{n+1}}\left(\tau^{*}_{n}\right) + Y^{\iota^{*}_{n+1}}_{\tau^{*}_{n}}
	\end{equation}
	
	Note that $\iota^{*}_{n+1}$ is now $\mathcal{F}_{\tau^{*}_{n}}$-measurable since $\tau^{*}_{n} = \tau^{*}_{n+1}$. Suppose that $\iota^{*}_{n-1} = i$, $\iota^{*}_{n} = j$ and $\iota^{*}_{n+1} = k$ for any three modes $i,j,k \in \mathbb{I}$ which necessarily satisfy $i \neq j$ and $j \neq k$ by definition of $\lbrace\iota^{*}_{n}\rbrace$. Substituting for $Y^{\iota^{*}_{n}}_{\tau^{*}_{n}}$ in \eqref{eq:AdmissibilityLemmaContradiction} and using condition~\eqref{assumption:Discrete-Time-SwitchingCosts-No-Arbitrage} for the switching costs gives,
	\[
	Y^{i}_{\tau^{*}_{n}} = -\gamma_{i,j}\left(\tau^{*}_{n}\right) -\gamma_{j,k}\left(\tau^{*}_{n}\right) + Y^{k}_{\tau^{*}_{n}} < -\gamma_{i,k}\left(\tau^{*}_{n}\right) + Y^{k}_{\tau^{*}_{n}}.
	\]	
	The previous arguments have just shown
	\[
	-\gamma_{i,k}\left(\tau^{*}_{n}\right) + Y^{k}_{\tau^{*}_{n}} > -\gamma_{i,j}\left(\tau^{*}_{n}\right) + Y^{j}_{\tau^{*}_{n}}= \max\limits_{l \neq i} \left\lbrace -\gamma_{i,l}\left(\tau^{*}_{n}\right) + Y^{l}_{\tau^{*}_{n}} \right\rbrace
	\]
	which is a contradiction for every $k \in \mathbb{I}$. Since $i \neq j$ and $j \neq k$ were arbitrary modes, for $n \ge 1$ it is true that $\mathsf{P}\left(\{\tau^{*}_{n} < T \} \cap \{\tau^{*}_{n} = \tau^{*}_{n+1}\}\right) = 0$.
\end{proof}

\section*{Acknowledgments}
The author would also like to acknowledge all those whose comments led to this improved version of the manuscript, including his PhD supervisor J. Moriarty and colleague T. De Angelis at the University of Manchester.



\end{document}